\documentclass[]{amsart}
\usepackage{color}
\usepackage{amsmath,amscd,amsthm,upref}
\usepackage[psamsfonts]{amssymb}
\usepackage[all]{xy}
\usepackage{enumerate}

\newtheorem{thm}{Theorem}[]

\newtheorem{lmm}[thm]{Lemma}
\newtheorem{prp}[thm]{Proposition}
\theoremstyle{definition}
\newtheorem{dfn}[thm]{Definition}

\theoremstyle{remark}
\newtheorem*{rmk}{Remark}

\DeclareMathOperator{\Diff}{\bf Diff}
\DeclareMathOperator{\Top}{\bf Top}

\newcommand{\bI}{\partial I}
\newcommand{\abs}[1]{\lvert{#1}\rvert}

\title{On homotopy types of diffeological cell complexes}
\author[T. Haraguchi]{Tadayuki Haraguchi} 
\address
{Faculty of Education for Human Growth,
Nara Gaken University,
Nara 636-8503, Japan}
\email{t-haraguchi@naragakuen-u.jp}
\author[K. Shimakawa]{Kazuhisa Shimakawa}
\address
{Graduate School of Natural Science and Technology,
Okayama University,
Okayama 700-8530, Japan}
\email{kazu@math.okayama-u.ac.jp}
\date{\today}
\begin{document}
%<<<
\maketitle
%{\color{red}\maketitle}
%>>>
%%%%%%%%%%%%%%%%%%%%%%%%%% abstract
\begin{abstract}
% In this paper,
%<<<
% we introduce the notion of smooth cell complexes (cf.~Definition \ref{dfn:relative cell complex}) within the category of
% diffeological spaces,
% and we shall give the following properties.
{%\color{red} 
We introduce the notion of smooth cell complexes
  and its subclass consisting of gathered cell complexes within the category of diffeological spaces
  (cf.~Definitions \ref{dfn:relative cell complex} and \ref{dfn:gathered cell complex}).
  It is shown that the following hold.}
%>>>

\begin{enumerate}
\item
%<<<
% Any smooth cell complex induces a topological cell complex equipped with paracompactness and Hausdorffness 
% (cf.~Proposition \ref{prp:topological cell complex} and Proposition \ref{prp:paracompactum}).
% {\color{red} With respect to the $D$-topology 
{%\color{red}
  With respect to the $D$-topology,
  every smooth cell complex is a topological cell complex (cf.~Proposition~\ref{prp:topological cell complex}).
  It is paracompact and Hausdorff if it is countable (cf.~Proposition~\ref{prp:paracompactum}).}
%>>>
\item
Every continuous map between gathered cell complexes is continuously homotopic to a smooth map (cf.~Theorem \ref{apprximate on cell}).
% that is,
% this property means extending the notion of Whitney approximation on smooth manifolds to smooth cell complexes.
% %The Whiteny approximation on smooth cell complexes (cf.~Theorem \ref{apprximate on cell}).
% {\color{red}
%   This can be regarded as an analogy of the Whitney approximation on smooth manifolds.
% }
\item
Any topological cell complex is 
continuously homotopy equivalent to a gathered (hence smooth) cell complex (cf.~Theorem \ref{thm:smooth up}).
\item
Every $D$-open cover of a smooth countable cell complex has 
a subordinate partition of unity by smooth functions (cf.~Theorem \ref{thm:partition of unity}).
%A partition of unity on smooth cell complexes (cf.~Theorem \ref{thm:partition of unity}).
\end{enumerate}
\end{abstract}
%%%%%%%%%%%%%%%%%%%%%%%%%%%%%%%%%%%%%%%%%%%%%%%%%%%%%%%%%%%%%%%%%%%%%%
%%%%%%%%%%%%%%%%%%%%%%
%%%%%%%%%%%%%%%%%%%%%% relative cell complexes
%%%%%%%%%%%%%%%%%%%%%%
%%%%%%%%%%%%%%%%%%%%%%%%%%%%%%%%%%%%%%%%%%%%%%%%%%%%%%%%%%%%%%%%%%%%%%
\section{Smooth relative cell complexes}
In this section we give the notion of smooth cell complexes and its fundamental properties.
To make the statement more precise, we shall work in the category $\Diff$
of diffeological spaces. As described in \cite{Zem}, the category $\Diff$ is complete,
cocomplete and cartesian closed. Moreover, there is a left adjoint functor $T$
from $\Diff$ to the category $\Top$ of topological spaces which assigns to every
diffeological space $X$ its underlying space equipped with $D$-topology (cf.~{\cite[Proposition 2.1]{SYH}}).
Throughout the paper,
$\mathbf{R}^{n}$ denotes the $n$-dimensional Euclidean space equipped with the standard diffeology consisting of 
all smooth parametrizations of $\mathbf{R}^{n}$.
Denote by $\partial I^{n}$ the boundary of the $n$-dimensional cube $I^{n}$,
and let
\[
L^{n-1}= \partial I^{n-1} \times I \cup I^{n-1} \times \{0\}\ \mbox{and}\
J^{n-1}= \partial I^{n-1} \times I \cup I^{n-1} \times \{1\}
\]
for each $n \geq 1$.
We regard $I^{n},\ \partial I^{n},\ L^{n-1}$ and $J^{n-1}$ as subspaces of $\mathbf{R}^{n}$.

%<<<
% We now introduce infinite gluing construction similar to that of {\cite[Definition 2.1.9]{Hov}}.
% Suppose $\mathcal C$ is a category with all small colimits,
% and $\delta$ is an ordinal.
{%\color{red}
  Let $\mathcal C$ be a category with all small colimits and $\delta$ an ordinal.
}
% >>>
A colimit preserving functor $Z \colon \delta \to \mathcal{C}$ 
is called a $\delta$-sequence (cf.~{\cite[Definition 2.1.1]{Hov}}) if
for all limit ordinals $\beta < \delta$,
the induced map
\[
\mbox{colim}_{\alpha < \beta} Z_{\alpha} \to Z_{\beta}
\]
is an isomorphism.
% In {\cite[Chapter 2.4]{Hov}},
% we will introduce topological cell complexes defined by a $\delta$-sequence $Z \colon \delta \to \bf Top$.

Given two smooth maps $f \colon Y \to Z$ and $g \colon Y^{\prime} \to Z$ between diffeological spaces,
we denote by $ f \cup g$ the composition
\[
\bigtriangledown \circ (f \coprod g) \colon Y \coprod Y^{\prime} \to Z \coprod Z \to Z,
\]
where $\bigtriangledown$ is the foloding map of $Z \coprod Z$ onto $Z$.
%Let $\mathcal I$ be the set of maps consisting of the boundary inclusions $\partial I^{n} \to I^{n}$ for all $n \geq 0$.
\begin{dfn}[{\cite[Definition 5.12]{HS}}]\label{dfn:relative cell complex}
Let $(X,A)$ be a pair of diffeological spaces.
We say that $(X,A)$ is a smooth relative cell complex if there is an ordinal $\delta$ and 
a $\delta$-sequence $Z \colon \delta \to \Diff$ such that the inclusion $i \colon A \to X$ coincides with the 
composition $Z_{0} \to \mbox{colim}Z$ and that for each $\beta$ such that $\beta+1<\delta$,
there exists a pushout square
\[
\xymatrix{
\partial I^{n_{\beta+1}}
\ar[r]^{\phi_{\beta+1}}
\ar[d]
&
Z_{\beta}
\ar[d]^{i_{\beta+1}}
\\
I^{n_{\beta+1}}
\ar[r]_{\Phi_{\beta+1}}
&
Z_{\beta+1},
}
\]
that is,
$Z_{\beta+1}$ is an adjunction space $Z_{\beta} \cup_{\phi_{\beta+1}} I^{n_{\beta+1}}$.
Then $i_{\beta+1} \cup \Phi_{\beta+1}$ coincides with the subduction from $Z_{\beta} \coprod I^{n_{\beta+1}}$ to $Z_{\beta+1}$.
We call
$\phi_{\beta+1}$ and $\Phi_{\beta+1}$ attaching and characteristic maps,
respectively.
We say that $X$ is a smooth cell complex if $A$ is the empty set.
\end{dfn}

%\begin{rmk}
%We will define topological relative cell complexes by 
%the same way as Definition \ref{dfn:relative cell complex}.
%In {\cite[Chapter 2.4]{Hov}},
%we discuss them called $I^{\prime}$-cell
%when giving a cofibrantly generated model structure on $\Top$.
%\end{rmk}
\begin{rmk}
%<<<
% N.~Iwase and N.~Izumida intoroduce smooth CW complexes defined characteristic maps by disks in \cite{Iwase}.
% H.~Kihara presents the notion of smooth simplicial cell complexes when giving the model structure of diffeological spaces in \cite{Kihara}.
{%\color{red} 
  N.~Iwase and N.~Izumida \cite{Iwase} intoroduces smooth CW complexes by using characteristic maps
  from the unit disks instead of cubes.
  On the other hand, H.~Kihara \cite{Kihara} uses the notion of smooth simplicial cell complexes when discussing
  the model structure of diffeological spaces.}
%>>>
\end{rmk}

%%%%%%%%%%%%%%%%%%%%%%%%%%%%%%%%%% subspace %%%%%%%%%%%%%%%%%%%%%%%%%%%%%%%
%%%%%%%%%%%%%%%%%%%%%%%%%%%%%%%%%%%%%%%%%%%%%%%%%%%%%%%%%%%%%%%%%%%%%%%%%%
{%\color{red}%
  Let $T \colon \Diff \to \Top$ be the functor which assigns to every
  diffeological space the underlying space having the $D$-topology
  associated with $X$.  }

\begin{prp}\label{prp:topological cell complex}
Let $(X,A)$ be a smooth relative cell complex given by 
a $\delta$-sequence $Z \colon \delta \to \Diff$.
% where $\delta$ is an ordinal.
Then its image $(TX, TA)$ under $T \colon \Diff \to \Top$ is a topological relative cell complex
%<<<
{%\color{red} 
in the sense of {\cite[Chapter 2.4]{Hov}}.}
%>>>
% In particular,
%$TZ_{\beta}$ is a subspace of $TZ_{\beta+1}$ for each $\beta$ such that $\beta+1 < \delta$.
\end{prp}
\begin{proof}
For each $\beta$ such that $\beta+1 < \delta$,
the pair $(TI^{n_{\beta+1}}, T\partial I^{n_{\beta+1}})$ is homeomorphic to 
the pair $(I^{n_{\beta+1}}, \partial I^{n_{\beta+1}})$ of topological subspaces of $\mathbf{R}^{n_{\beta+1}}$
by {\cite[Lemma 3.17]{DGE}}.
Since the functor $T$ preserves colimits,
we have a pushout square
\[
\xymatrix{
\partial I^{n_{\beta+1}}
\ar[d]
\ar[r]^{T\phi_{\beta+1}}
&
TZ_{\beta}
\ar[d]^{Ti_{\beta+1}}
\\
I^{n_{\beta+1}}
\ar[r]_{T\Phi_{\beta+1}}
&
TZ_{\beta+1}.
}
\]
Therefore, % we have the $\delta$-sequence
$TZ \colon \delta \to \Top$ is a $\delta$-sequence such that 
the inclusion $T i \colon TA \to TX$ coincides with the composition $TZ_{0} \to \mbox{colim}TZ$.
% since the functor $T$ preserves the colimits.
\end{proof}

\begin{rmk}
% Let $A$ be a subspace of a diffeological space $X$.
% In general,
% the $D$-topological space $TA$ is not always homeomorphic to 
% the topological subspace $A$ of $TX$ (cf.~{\cite[3.3]{DGE}}).
% %Since $TZ_{\beta}$ is a topological subspace of $TX$ 
% In Proposition \ref{prp:topological cell complex}, for each $\beta < \delta$ the $D$-topological space $TZ_{\beta}$ 
% is the subspace of $TX$ by the properties of pushouts.
% Thus we can discuss the $D$-topology on smooth relative cell complexes as well as the case of 
% topological relative cell complexes.
Let $A$ be a subspace of a diffeological space $X$.
In general,
the $D$-topological space $TA$ is not a topological subspace of $TX$ (cf.~{\cite[3.3]{DGE}}).
{%\color{red} 
But in the situation of} Proposition~\ref{prp:topological cell complex}, the $D$-topological space $TZ_{\beta}$ 
is a subspace of $TX$,
for every $\beta < \delta$. % by the properties of pushouts.
Thus we can discuss the $D$-topology on smooth relative cell complexes just as in the case of 
topological relative cell complexes.
\end{rmk}

%<<<
% In order to discuss our smooth approximate structures,
% we will introduce the following tameness.
% \begin{dfn}[{\cite[Definition 3.8]{HS}}]\label{dfn:definition of tame maps}
% Let $f \colon K \to X$ be a smooth map from a cubical subcomplex $K$
% of $I^n$ (e.g.\ $I^n$, $\bI^n$, $L^{n-1}$ or $J^{n-1}$) to a diffeological
% space $X$.  Suppose $0 < \epsilon \leq 1/2$ and $\alpha \in \{0,1\}$.
% Then $f$ is called to
% be \emph{$\epsilon$-tame} if we have
% \[
% f(t_1,\cdots,t_{j-1},t_j,t_{j+1},\cdots,t_n) =
% f(t_1,\cdots,t_{j-1},\alpha,t_{j+1},\cdots,t_n)
% \]
% for every $(t_1,\cdots,t_n) \in K$ such that $\abs{t_j - \alpha} \leq \epsilon$ holds.
% Moreover a smooth homotopy $H \colon Y \times I \to Z$ between diffeological spaces is called 
% $\epsilon$-\emph{tame} if we have
% \[
% H(y,t)=H(y,\alpha)
% \]
% for each $(y,t) \in Y \times I$ such that $|t - \alpha|\leq \epsilon$.
% %
% We use the abbreviation ``tame'' to mean $\epsilon$-tame for some
% $\epsilon > 0$.
% \end{dfn}
{%\color{red}%
  Recall from \cite[Definition 3.8]{HS} that a smooth map from a cubical
  subcomplex $K$ of $I^n$ (e.g.\ $I^n$, $\bI^n$, $L^{n-1}$ or $J^{n-1}$) to
  a diffeological space $X$ is called to be \emph{$\epsilon$-tame} if we
  have
  \[
    f(t_1,\cdots,t_{j-1},t_j,t_{j+1},\cdots,t_n) =
    f(t_1,\cdots,t_{j-1},\alpha,t_{j+1},\cdots,t_n)
  \]
  for every $(t_1,\cdots,t_n) \in K$ and $\alpha \in \{0,1\}$ such that
  $\abs{t_j - \alpha} \leq \epsilon$ holds.  Moreover a smooth homotopy
  $H \colon Y \times I \to Z$ between diffeological spaces is called
  $\epsilon$-\emph{tame}, where $0 < \epsilon \leq 1/2$, if we have
  \[
    H(y,t)=H(y,\alpha)
  \]
  for each $(y,t) \in Y \times I$ such that $|t - \alpha|\leq \epsilon$.
  We use the abbreviation ``tame'' to mean $\epsilon$-tame for some
  $\epsilon > 0$.
}
%>>>

% Then we have the approximate homotopy extension property on cubes.
%\begin{prp}[{\cite[Proposition 3.13]{HS}}]\label{prp:tame maps are extendable}
%Any $\epsilon$-tame map $f \colon J^{n-1} \to X$ can be extended to
%a $\sigma$-tame map $g \colon I^n \to X$ for any
%$\sigma < \epsilon$.
%
%Moreover, if $f$ is $\epsilon'$-tame on $\bI^{n-1} \times \{0\}$
%then $g$ can be taken to be $\sigma'$-tame on $I^{n-1} \times \{0\}$
%for any $\sigma' < \epsilon'$.  
%\end{prp}

{%\color{red}%
  \begin{dfn}\label{dfn:gathered cell complex}
    A smooth relative cell complex is called to be \emph{gathered} if its all attaching maps are tame.
  \end{dfn}
}

{%\color{red}%
  The following proposition plays a key role in the proofs of our main results.  }

\begin{prp}[{\cite[Proposition 6.9]{HS}}]\label{approximation on cube}
Let $X$ be a gathered cell complex,
and let $L$ be a cubical subcomplex of $I^{n}$.
Let $f$ be a continuous map form $I^{n}$ to $TX$ such that 
its ristriction to $L$ is a smooth tame map.
Then there exists a tame map $g \colon I^{n} \to X$ such that 
$f$ and $Tg$ are continuously homotopic relative to $L$.
\end{prp}
%%%%%%%%%%%%%%%%%%%%%%%%%%%%%%%%%%%%%%%%%%%%%%%%%%%%%%%%%%%%%%%%%%%%%%%%%%%%%%%%%%%%
%%%%%%%%%%%%%%%%%%%%%%%%%%%%%%%%%%%%%%%%%%%%%%%%%%%%%%%%%%%%%%%%%%%%%%%%%%%%%%%%%%%
%
% Whitney approximation on smooth cell complexes
%
%%%%%%%%%%%%%%%%%%%%%%%%%%%%%%%%%%%%%%%%%%%%%%%%%%%%%%%%%%%%%%%%%%%%%%%%%%%%%%%%%%
%%%%%%%%%%%%%%%%%%%%%%%%%%%%%%%%%%%%%%%%%%%%%%%%%%%%%%%%%%%%%%%%%%%%%%%%%%%%%%%%%%%%
\section{The Whitney apprximation on smooth cell complexes}
The classical Whitney approximation theorem states that any continuous map between smooth manifolds is homotopic to a smooth map (cf.~{\cite[Theorem 6.19]{Lee}}).
In this section we will extend this theorem on gathered cell complexes.

\begin{thm}\label{apprximate on cell}
Let $(X,A)$ and $(Y,B)$ be relative gathered cell complexes.
Let $f $ be a continuous map from $(TX,TA)$ to $(TY,TB)$ such that its restriction to $A$ is a smooth map.
Then there exists a smooth map $g \colon (X,A) \to (Y,B)$ such that $f$ and $Tg$ are continuously homotopic relative to $TA$.
\end{thm}

\begin{proof}
Let $(X,A)$ be a relative gathered cell complex given by a $\delta$-sequence $Z \colon \delta \to \Diff$.
Starting from the trivial homotopy of $g_{0}=f|A$,
%for each $\beta$ such that $\beta+1 < \delta$ 
we inductively construct a smooth map $g_{\beta} \colon (Z_{\beta},A) \to (Y,B)$ and 
a continuous homotopy $h_{\beta} \colon TZ_{\beta} \times I \to TY$ between $f|TZ_{\beta}$ and $Tg_{\beta}$ relative to $A$.

For each $\beta$ such that $\beta+1<\delta$,
let $\phi_{\beta+1} \colon \partial I^{n_{\beta+1}}\to Z_{\beta}$ and 
$\Phi_{\beta+1} \colon I^{n_{\beta+1}} \to Z_{\beta+1}$ be an attaching and a characteristic maps,
respectively.
Suppose we have a smooth map $g_{\beta} \colon Z_{\beta} \to Y$ and a continuous tame homotopy 
$h_{\beta} \colon TZ_{\beta} \times I \to TY$ between $f|TZ_{\beta}$ and $Tg_{\beta}$ relative to $A$.
Then put 
$G_{\beta}=h_{\beta} \circ (\phi_{\beta+1} \times 1) \cup f \circ \Phi_{\beta+1} \colon L^{n_{\beta+1}} \to TY$.
Let us take such $\epsilon >0$ that $\phi_{\beta+1}$ and $h_{\beta}$ are $\epsilon$-tame,
and let $R_{\epsilon} \colon I^{n_{\beta+1}} \times I \to L^{n_{\beta+1}}$ be an $\epsilon$-approximate retraction 
(cf.~{\cite[Lemma 3.12]{HS}}) which is $\sigma$-tame ($\sigma < \epsilon$).
Then the composition $G_{\beta} \circ R_{\epsilon} \colon I^{n_{\beta+1}} \times I \to TY$
coincides with $G_{\beta}$ on $\partial I^{n_{\beta+1}} \times I$.
But then,
there exist by Proposition \ref{approximation on cube} a tame map $G^{\prime} \colon I^{n_{\beta+1}} \times I \to Y$
and a continuous tame homotopy $K \colon I^{n_{\beta+1}} \times I \times I \to TY$ between 
$G_{\beta} \circ R_{\epsilon}$ and $TG^{\prime}$ relative to $\partial I^{n_{\beta+1}} \times \{1\}$.
In fact,
$K(v,1,t)=g_{\beta}(\phi_{\beta+1}(v))$ holds for $(v,1,t) \in \partial I^{n_{\beta+1}} \times \{1\} \times I$.
Thus we get a $\sigma$-tame $g^{\prime}=G^{\prime}|I^{n_{\beta+1}} \times \{1\} \colon I^{n_{\beta+1}} \to Y$ and 
a tame homotopy $\tilde{G} \colon I^{n_{\beta+1}} \times I \to TY$ from $f \circ \Phi_{\beta+1}$ to $g^{\prime}$ 
given by the formula
\[
\tilde{G}(s,t)=
\left\{
\begin{array}{lll}
%f(\Phi_{\beta+1}(k(s,2t/\sigma)))
%&
%0 \leq t \leq \sigma/2
%\\
G_{\beta}\circ R_{\epsilon}(s,t)
&
0 \leq t \leq 1-\sigma/2
\\
K(s,1,(2t-2+\sigma)/\sigma)
&
1-\sigma/2 \leq t \leq 1.
\end{array}
\right.
\]
Now,
we have a diagram
\[
\xymatrix@C=50pt{
(TZ_{\beta} \times I) \coprod (I^{n_{\beta+1}} \times I)
\ar[r]^(0.7){h_{\beta} \cup \tilde{G}}
\ar[d]
&
TY
\\
TZ^{\beta+1} \times I.
\ar@{.>}[ru]
}
\]
Since the vertical arrow is a quotient map,
there exists a continuous tame homotopy $h_{\beta+1} \colon TZ_{\beta+1} \times I \to TY$ making the diagram
above commutative.
Clearly,
$h_{\beta+1}$ gives $f|TZ_{\beta+1} \simeq Tg_{\beta+1}$ rel $TA$,
where $g_{\beta+1} \colon Z_{\beta+1} \to Y$ is a smooth tame map induced by the map 
$
g_{\beta} \cup g^{\prime} \colon Z_{\beta} \coprod I^{n_{\beta+1}} \to Y.
$

%<<<
% Let $\beta$ be a limit ordinal of $\delta$.
% Then $Z_{\beta} \cong \mbox{colim}_{\alpha < \beta} Z_{\alpha}$ holds.
%>>>
{%\color{red}%
  Suppose now that $\beta$ is a limit ordinal of $\delta$.  Then we have
  $Z_{\beta} \cong \mbox{colim}_{\alpha < \beta} Z_{\alpha}$.  }%
For each $\alpha < \beta$, suppose there exists a smooth map
$g_{\alpha} \colon Z_{\alpha} \to Y$ and a continuous homotopy
$F_{\alpha} \colon TZ_{\alpha} \times I \to Y$ between $f|TZ_{\alpha}$ and
$Tg_{\alpha}$ such that for each $\tilde{\alpha} < \alpha$,
\[
g_{\alpha}|Z_{\tilde{\alpha}}=g_{\tilde{\alpha}}\colon Z_{\tilde{\alpha}} \to Y\ 
\mbox{and}\ 
F_{\alpha}|TZ_{\tilde{\alpha}} \times I=F_{\tilde{\alpha}} \colon TZ_{\tilde{\alpha}} \times I \to TY
\]
hold.
Since the functor $T$ preserves the colimits,
we have
\[
\mbox{colim}_{\alpha<\beta}(TZ_{\alpha} \times I) \cong (\mbox{colim}_{\alpha < \beta} TZ_{\alpha}) \times I
%\cong T(\mbox{colim}_{\alpha < \beta} Z_{\alpha}) \times I 
\cong TZ_{\beta} \times I.
\]
Thus we get a smooth map $g_{\beta}=\mbox{colim}_{\alpha<\beta}g_{\alpha} \colon Z_{\beta} \to Y$ and 
a continuous homotopy $F_{\beta}=\mbox{colim}_{\alpha<\beta} F_{\alpha} \colon TZ_{\beta} \times I \to TY$
between $f|TZ_{\beta}$ and $Tg_{\beta}$.
This completes the induction step and proves Theorem \ref{apprximate on cell}.
\end{proof}
%%%%%%%%%%%%%%%%%%%%%%%%%%%%%%%%%%%%%%%%%%%%%%%%%%%%%%%%%%%%%%%%%%%%%%%%%%%%%
%%%%%%%%%%%%%%%%%%%%%%%%%%%%%%%%%%%%%%%%%%%%%%%%%%%%%%%%%%%%%%%%%%%%%%%%%%%%%%%%%
%
% Smooth up to homotopy of topological cell complexes
%
%%%%%%%%%%%%%%%%%%%%%%%%%%%%%%%%%%%%%%%%%%%%%%%%%%%%%%%%%%%%%%%%%%%%%%%%%%%%%
%%%%%%%%%%%%%%%%%%%%%%%%%%%%%%%%%%%%%%%%%%%%%%%%%%%%%%%%%%%%%%%%%%%%%%%%%%%%%
%<<<
% \section{Smooth up to homotopy of topological cell complexes}
% In this section we give one result of Theorem A.1 introduced in {\cite[Appendix A]{Iwase}}.
{%\color{red}%
  \section{Smoothing of topological cell complexes}
  % By Proposition~\ref{prp:topological cell complex}, every smooth cell complex has 
  In this section we show that every topological cell complex is
  continuously homotopy equivalent to a smooth cell complex.  (Compare
  \cite[Theorem A.1]{Iwase}.)  }
%>>>
\begin{thm}\label{thm:smooth up}
Let $X^{\prime}$ be a topological cell complex given by a $\delta$-sequence $Z^{\prime} \colon \delta \to \Top$.
% where $\delta$ is an ordinal.
Then there exists a gathered cell complex $X$ such that $X^{\prime}$ is 
continuously homotopy equivalent to $TX$.
\end{thm}
\begin{proof}
Since $Z_{0}$ is the empty set,
$Z_{1}$ is a point.
We shall inductively construct a gathered cell complex $Z_{\beta}$ and 
a continuous homotopy equivalence $f_{\beta} \colon Z^{\prime}_{\beta} \to TZ_{\beta}$ such that 
$f_{\beta}| Z^{\prime}_{\alpha}=f_{\alpha} \colon Z^{\prime}_{\alpha} \to TZ_{\alpha}$ holds for each $\alpha < \beta$.

For each $\beta$ such that $\beta+1 < \delta$,
suppose we have a gathered cell complex $Z_{\beta}$ and a continuous homotopy 
equivalence $f_{\beta} \colon Z^{\prime}_{\beta} \to TZ_{\beta}$.
Then there exist homotopies $H^{\prime}_{\beta} \colon Z^{\prime}_{\beta} \times I \to Z^{\prime}_{\beta}$ 
and $H_{\beta} \colon TZ_{\beta} \times I \to TZ_{\beta}$ satisfying
$
1 \simeq g_{\beta}f_{\beta}
$
and
$
1 \simeq f_{\beta}g_{\beta}
$,
where $g_{\beta} \colon TZ_{\beta} \to Z^{\prime}_{\beta}$ is a homotopy inverse.
Then we have the following commutative diagram consisting of pushout squares
\[
\xymatrix@C=50pt{
\partial I^{n_{\beta+1}}
\ar[r]^{\phi^{\prime}_{\beta+1}}
\ar[d]
&
Z^{\prime}_{\beta}
\ar[r]^{f_{\beta}}
\ar[d]
&
TZ_{\beta}
\ar[d]
\\
I^{n_{\beta+1}}
\ar[r]_{\Phi_{\beta+1}^{\prime}}
&
Z^{\prime}_{\beta+1}
\ar@<-0.5ex>[r]_(0.4){ \tilde{f}_{\beta}}
&
TZ_{\beta} \cup_{f_{\beta} \phi^{\prime}_{\beta+1}}I^{n_{\beta+1}},
\ar@<-0.5ex>@{.>}[l]_(0.6){\exists \tilde{g}_{\beta}}
}
\]
where $\phi^{\prime}_{\beta+1}$ and $\Phi^{\prime}_{\beta+1}$ are attaching and characteristic maps,
respectively.
There exists the homotopy inverse $\tilde{g}_{\beta}$ of $\tilde{f}_{\beta}$ by {\cite[7.5.7]{Brown}}.
Then we can construct homotopies $1 \simeq \tilde{g}_{\beta}\tilde{f}_{\beta}$ and $1 \simeq \tilde{f}_{\beta} \tilde{g}_{\beta}$ extending $H^{\prime}_{\beta}$ and $H_{\beta}$,
respectively.
By Proposition \ref{approximation on cube} there exists a tame map 
$\phi_{\beta+1} \colon \partial I^{n_{\beta+1}} \to Z_{\beta}$ satisfying
\begin{eqnarray}\label{homotopic}
f_{\beta} \circ \phi^{\prime}_{\beta+1} \simeq T\phi_{\beta+1} \colon \partial I^{n_{\beta+1}} \to TZ_{\beta}.
\end{eqnarray}
We define $Z_{\beta+1}$ by the following pushout square
\[
\xymatrix{
\partial I^{n_{\beta+1}}
\ar[d]
\ar[r]^{\phi_{\beta+1}}
&
Z_{\beta}
\ar[d]
\\
I^{n_{\beta+1}}
\ar[r]_{\Phi_{\beta+1}}
&
Z_{\beta+1}.
}
\]
Since the functor $T$ preserves colimits,
and since the pair $(TI^{n_{\beta+1}}, T\partial I^{n_{\beta+1}})$ is homeomorphic to 
the pair $(I^{n_{\beta+1}}, \partial I^{n_{\beta+1}})$ by {\cite[Lemma 3.17]{DGE}},
the adjunction space $TZ_{\beta+1}$ is obtained by gluing $I^{n_{\beta+1}}$ to $TZ_{\beta}$ along an attaching map 
$T\phi_{\beta+1}$.
Then there is a homotopy equivalence
\[
\iota \colon TZ_{\beta} \cup_{f_{\beta}\phi^{\prime}_{\beta+1}}I^{n_{\beta+1}} \to TZ_{\beta+1}\ \mbox{rel} \ TZ_{\beta}
\]
by the conditions {\cite[7.5.5]{Brown}} and  (\ref{homotopic}).
Let 
$
\iota^{\prime} \colon TZ_{\beta+1} \to TZ_{\beta} \cup_{f_{\beta}\phi^{\prime}_{\beta+1}}I^{n_{\beta+1}}
$
be the homotopy inverse of $\iota$.
Let us define homotopy equivalences $f_{\beta+1} \colon Z^{\prime}_{\beta+1} \to TZ_{\beta+1}$ and 
$g_{\beta+1} \colon TZ_{\beta+1} \to Z^{\prime}_{\beta+1}$ by the compositions 
$\iota \circ \tilde{f}_{\beta}$ and $\tilde{g}_{\beta} \circ \iota^{\prime}$,
respectively.
\[
\xymatrix@C=50pt{
Z^{\prime}_{\beta+1}
\ar@<-0.5ex>[r]_(0.4){\tilde{f}_{\beta}}
&
TZ_{\beta} \cup_{f_{\beta}\phi^{\prime}_{\beta+1}}I^{n_{\beta}}
\ar@<-0.5ex>[r]_(0.6){\iota}
\ar@<-0.5ex>[l]_(0.6){\tilde{g}_{\beta}}
&
TZ_{\beta+1}
\ar@<-0.5ex>[l]_(0.39){\iota^{\prime}}
}
\]
Then we can construct homotopies 
\[
H^{\prime}_{\beta+1}\colon Z^{\prime}_{\beta+1}\times I \to Z^{\prime}_{\beta+1}\ \mbox{and}\ 
H_{\beta+1}\colon TZ_{\beta+1} \times I \to TZ_{\beta+1}
\]
satisfying 
$1 \simeq g_{\beta+1}f_{\beta+1}$ and $1 \simeq f_{\beta+1}g_{\beta+1}$
which can be extended to $H^{\prime}_{\beta}$ and $H_{\beta}$,
respectively.

Let $\beta$ be a limit ordinal of $\delta$.
Suppose for each $\alpha < \beta$,
there exist a gathered cell complex $Z_{\alpha}$ and a continuous homotopy equivalence 
$f_{\alpha} \colon Z^{\prime}_{\alpha} \to TZ_{\alpha}$ such that the restriction $f_{\alpha}|Z^{\prime}_{\tilde{\alpha}}$
is a homotopy equivalence $f_{\tilde{\alpha}} \colon Z^{\prime}_{\tilde{\alpha}} \to TZ_{\tilde{\alpha}}$ 
for each $\tilde{\alpha} < \alpha$.
Then we have a gathered cell complex $Z_{\beta}$ defined by $\mbox{colim}_{\alpha<\beta}Z_{\alpha}$ 
and a continuous homotopy equivalence 
$f_{\beta}=\mbox{colim}_{\alpha < \beta} f_{\alpha} \colon Z^{\prime}_{\beta} \to TZ_{\beta}$.

Therefore we have a gathered cell complex $X=\mbox{colim}Z$ given by 
the $\delta$-sequence $Z \colon \delta \to \Diff$ and a continuous homotopy equivalence
$\mbox{colim}f_{\beta} \colon X^{\prime} \to TX$.
\end{proof}
%%%%%%%%%%%%%%%%%%%%%%%%%%%%%%%%%%%%%%%%%%%%%%%%%%%%%%%%%%%%%%%%%%%%%%%%%%%%%
%%%%%%%%%%%%%%%%%%%%%%%%%%%%%%%%%%%%%%%%%%%%%%%%%%%%%%%%%%%%%%%%%%%%%%%%%%%%%
%
%partition of unity
%
%%%%%%%%%%%%%%%%%%%%%%%%%%%%%%%%%%%%%%%%%%%%%%%%%%%%%%%%%%%%%%%%%%%%%%%%%%%%%
%%%%%%%%%%%%%%%%%%%%%%%%%%%%%%%%%%%%%%%%%%%%%%%%%%%%%%%%%%%%%%%%%%%%%%%%%%%%%
%%%%%%%%%%%%%%%%%%%%%%%%%%%%%%%%%%%%%%%%%%%%%%%%%%%%%%%%%%%%%%%%%%%%%%
\section{Partition of unity on smooth cell complexes}
%<<<
% In this section we introduce a partition of unity on smooth cell complexes.
% First,
% We introduce a partition of unity a partition of unity on diffeological spaces which is similarly defined in the case of topological spaces (cf.~\cite{H1}).
{%\color{red}%
  In this section we show that every smooth cell complex admits a partition of unity by smooth functions.
}
% >>>

Let $X$ be a diffeological space.
If $\psi \colon X \to \mathbf{R}$ is a real-valued smooth map,
the support of $\psi$,
denoted by supp$\psi$,
{%\color{red} 
is the closure, with respect to the $D$-topology, of the set of points such that $\psi$ is non-zero:
%>>>
\[
\mbox{supp} \psi =
\overline{
\{x \in X \mid \psi(x) \not =0 \}
}.
\]
% where the closure is given by $D$-open sets of $X$.
}
A collection of subsets of $X$ is called locally finite if each $x \in X$ has a $D$-open neighbourhood that intersects
{%\color{red} 
with only finitely many members of
of the collection.}
Let $\mathcal{U}=\{A_{\lambda}\}_{\lambda \in \Lambda}$ be an arbitary $D$-open cover of $X$.
We say that a collection $\{\psi_{\lambda} \colon X \to \mathbf{R}\mid \lambda \in \Lambda \}$ 
is a partition of unity subordinate to $\mathcal{U}$ if the following conditions are satisfied

\begin{enumerate}
\item
$0 \leq \psi_{\lambda}(x) \leq 1$ for all $\lambda \in \Lambda$ and all $x \in X$
\item
supp$\psi_{\lambda} \subset A_{\lambda}$ for all $\lambda \in \Lambda$
\item
the set $\{ \mbox{supp}\psi_{\lambda} \mid \lambda \in \Lambda \}$ of supports is locally finite and
\item
$\sum_{\lambda} \psi_{\lambda}(x)=1$ for all $x \in X$.
\end{enumerate}

In particular,
any diffeological subcartesian space has a partition of unity subordinate to arbitary $D$-open cover (cf.~{\cite[Theorem 3.3]{H1}).
%%%%%%%%%%%%%%%%%%%%%%%% main lemma %%%%%%%%%%%%%%%%%%%%%%%%%%%%%%%%%
%%%%%%%%%%%%%%%%%%%%%%%%%%%%%%%%%%%%%%%%%%%%%%%%%%%%%%%%%%%%%%%%%%%%
\begin{lmm}\label{smooth function}
Let $X$ be a smooth cell complex given by a $\delta$-sequence $Z \colon \delta \to \Diff$,
where $\delta$ is a countable ordinal.
Let $A$ and $B$ be $D$-closed and $D$-open sets of $X$,
respectively,
such that $A \subset  B$ holds.
Then there exists a smooth map $f \colon X \to I$ such that $f \equiv 1$ on $A$ and supp$f \subset B$.
\end{lmm}

\begin{proof}
Let $\beta_{0}$ be the minimum value $\beta< \delta$ satisfying $Z_{\beta} \cap A \not = \emptyset$.
If $\beta_{0}=0$ holds,
we define the constant map $f_{\beta_{0}} \colon Z_{\beta_{0}} \to I$ with the value $1 \in I$.
Let $\beta_{0}>0$ and let $\Phi_{\beta_{0}} \colon I^{n_{\beta_{0}}} \to Z_{\beta_{0}}$ be a characteristic map.
Then there exists an open set $M^{\prime}_{\beta_{0}}$ of $I^{n_{\beta_{0}}}$ such that
\[
\Phi_{\beta_{0}}^{-1}(A) \subset M^{\prime}_{\beta_{0}} \subset \overline{M^{\prime}}_{\beta_{0}} \subset 
\mbox{Int} I^{n_{\beta_{0}}} \cap \Phi_{\beta_{0}}^{-1}(B),
\]
since $I^{n_{\beta_{0}}}$ is a normal space.
There exists a smooth map $f^{\prime}_{\beta_{0}} \colon I^{n_{\beta_{0}}} \to I$ such that
\[
f^{\prime}_{\beta_{0}}|\overline{M^{\prime}}_{\beta_{0}} \equiv 1,\
\mbox{supp}f^{\prime}_{\beta_{0}} \subset \mbox{Int}I^{n_{\beta_{0}}} \cap \Phi^{-1}_{\beta_{0}}(B)
\]
by {\cite[Corollary 2.19]{Lee}}.
Since $i_{\beta_{0}} \cup \Phi_{\beta_{0}}$ is a subduction,
there exists a smooth map $f_{\beta_{0}} \colon Z_{\beta_{0}} \to I$ such that the following diagram commutes
\[
\xymatrix@C=50pt{
Z_{\beta_{0}-1} \coprod I^{n_{\beta_{0}}}
\ar[r]^(0.65){C_{0} \cup f^{\prime}_{\beta_{0}}}
\ar[d]_{i_{\beta_{0}} \cup \Phi_{\beta_{0}}}
&
I
\\
Z_{\beta_{0}},
\ar@{.>}[ru]_{\exists f_{\beta_{0}}}
}
\]
where $C_{0} \colon Z_{\beta_{0}-1} \to I$ is the constant map with the value $0 \in I$.
Let $M_{\beta_{0}}$ be a $D$-open set $\Phi_{\beta_{0}}(M^{\prime}_{\beta_{0}})$ of $Z_{\beta_{0}}$.
Then we have the following conditions
\begin{enumerate}
\item
$
Z_{\beta_{0}} \cap A \subset M_{\beta_{0}} \subset \overline{M}_{\beta_{0}}=\Phi_{\beta_{0}}(\overline{M^{\prime}}_{\beta_{0}})
\subset Z_{\beta_{0}} \cap B,
$
\item
$
f_{\beta_{0}} \equiv 1
$
on $\overline{M}_{\beta_{0}}$ and
\item
$\mbox{supp}f_{\beta_{0}} \subset Z_{\beta_{0}} \cap B$.
\end{enumerate}
Let $\beta_{0} < \beta$ and suppose there are a $D$-open set $M_{\beta}$ of $Z_{\beta}$ and a smooth map $f_{\beta} \colon Z_{\beta} \to I$ 
satisfying the following conditions
\begin{enumerate}
\item
$
Z_{\beta} \cap A \subset M_{\beta} \subset \overline{M}_{\beta} \subset Z_{\beta} \cap B,\
Z_{\beta -1} \cap M_{\beta}=M_{\beta-1},
$
\item
$f_{\beta}  \equiv 1$ on $\overline{M}_{\beta}$ and
\item
$\mbox{supp}f_{\beta} \subset Z_{\beta} \cap B$.
\end{enumerate}
Let $\phi_{\beta+1} \colon \partial I^{n_{\beta+1}} \to Z_{\beta}$ and $\Phi_{\beta+1} \colon I^{n_{\beta+1}} \to Z_{\beta+1}$ be attaching and characteristic maps,
respectively.
Then there exists a $D$-open set $M^{\prime}_{\beta+1}$ of $I^{n_{\beta +1}}$ satisfying
\[
\Phi^{-1}_{\beta+1}(A) \subset M^{\prime}_{\beta +1} \subset \overline{M^{\prime}}_{\beta+1} \subset \Phi^{-1}_{\beta+1}(B)\ \mbox{and} \ 
M^{\prime}_{\beta+1} \cap \partial I^{n_{\beta+1}} \subset \phi_{\beta+1}^{-1}(M_{\beta})
\]
since $I^{n_{\beta+1}}$ is a normal space.
Let us define a smooth map
\[
f^{\prime}_{\beta+1} \colon \partial I^{n_{\beta+1}} \cup \overline{M^{\prime}}_{\beta+1} \to I
\]
by the formula
\[
f^{\prime}_{\beta+1}(t)=
\left\{
\begin{array}{llll}
f_{\beta} \phi_{\beta+1}(t) & t \in \partial I^{n_{\beta+1}} \\
0 & t \in \overline{M^{\prime}}_{\beta+1}.
\end{array}
\right.
\]
Then,
by Tietze extension theorem (cf.~\cite{Urysohn}),
there exists a continuous function $f^{\prime \prime}_{\beta+1} \colon I^{n_{\beta+1}} \to I$ such that it extends $f^{\prime}_{\beta+1}$ and 
$\mbox{supp}f^{\prime \prime}_{\beta+1} \subset \Phi^{-1}_{\beta+1}(B)$ holds. 
Moreover there exists a smooth map $f^{\prime \prime \prime} \colon I^{n_{\beta+1}} \to I$ satisfying 
\[
f^{\prime \prime \prime} \simeq f^{\prime \prime}_{\beta+1} \ (\mbox{rel}\ \partial I^{n_{\beta+1}} \cup \overline{M^{\prime}}_{\beta+1}) \
\mbox{and}\ \mbox{supp}f^{\prime \prime \prime}_{f_{\beta+1}} \subset \Phi_{\beta+1}^{-1}(B),
\]
by the Whitney approximation theorem (cf.~{\cite[Theorem 6.19]{Lee}}).
Now we have a smooth map $f_{\beta+1} \colon Z_{\beta+1} \to I$ extending $f_{\beta}$ and making the following diagram commutative
\[
\xymatrix@C=50pt{
Z_{\beta} \coprod I^{n_{\beta+1}}
\ar[d]_{i_{\beta+1} \cup \Phi_{\beta+1}}
\ar[r]^(0.65){f_{\beta} \cup f_{\beta+1}^{\prime \prime \prime}}
&
I
\\
Z_{\beta+1}.
\ar@{.>}[ru]_{\exists f_{\beta+1}}
}
\]
Let $M_{\beta+1}$ be the adjunction space $M_{\beta} \cup_{\phi_{\beta+1}} M^{\prime}_{\beta+1}$.
Then it is $D$-open in $Z_{\beta+1}$ and the closure $\overline{M}_{\beta+1}$ is 
$\overline{M}_{\beta} \cup_{\phi_{\beta+1}} \overline{M^{\prime}}_{\beta+1}$.
Now,
we get the following conditions
\begin{enumerate}
\item
$
Z_{\beta+1} \cap A \subset M_{\beta +1} \subset \overline{M}_{\beta +1} \subset Z_{\beta +1} \cap B,\
Z_{\beta} \cap M_{\beta+1}=M_{\beta},
$
\item
$f_{\beta +1}  \equiv 1$ on $\overline{M}_{\beta+1}$ and
\item
$\mbox{supp}f_{\beta+1} \subset Z_{\beta+1} \cap B$.
\end{enumerate}
Therefore we can construct a smooth map $f \colon X \to I$ given by $\mbox{colim} f_{\beta}$.
Then it satisfies $f \equiv 1$ on $A$ and $\mbox{supp}f \subset B$.
\end{proof}
A topological space is called a paracompactum if it is a paracompact Hausdorff space.
We know that a topological CW complex is a paracompactum (cf.~{\cite[Theorem 1.3.5]{FP}}).
Now we have the following.
\begin{prp}\label{prp:paracompactum}
Let $X$ be a smooth cell complex given by a $\delta$-sequence $Z \colon \delta \to \Diff$,
where $\delta$ is a countable ordinal.
Then $TX$ is a paracompactum.
\end{prp}
\begin{proof}
Suppose $Z_{\beta}$ is a paracompactum for each $\beta \geq 0$.
Since tha functor $T$ preserves the colimits and the pair
$(T I^{n_{\beta+1}}, T \partial I^{n_{\beta+1}})$ is homeomorphic to the pair $(I^{n_{\beta+1}},\partial I^{n_{\beta+1}})$
by {\cite[Lemma 3.17]{DGE}},
we have the following pushout square
\[
\xymatrix{
\partial I^{n_{\beta+1}}
\ar[d]
\ar[r]^{T\phi_{\beta+1}}
&
TZ_{\beta}
\ar[d]^{Ti_{\beta+1}}
\\
I^{n_{\beta+1}}
\ar[r]_{T\Phi_{\beta+1}}
&
TZ_{\beta+1}.
}
\]
Then $TZ_{\beta+1}$ is a paracompactum by {\cite[Exercise 7.18]{James}}.
Therefore we get a paracompactum $X=\mbox{colim}Z$ by {\cite[Proposition A.5.1]{FP}}.
\end{proof}
%%%%%%%%%%%%%%%%%%%%%%%%%%%%%%%%%%%%%%%%%%%%%%%%%
%%%%%%%% main theorem partition of unity
%%%%%%%%%%%%%%%%%%%%%%%%%%%%%%%%%%%%%%%%%%%%%%%%%%
\begin{thm}\label{thm:partition of unity}
Let $X$ be a smooth cell complex given by a $\delta$-sequence $Z \colon \delta \to \Diff$,
where $\delta$ is a countable ordinal.
Then for any $D$-open cover $\mathcal{U}=\{U_{\lambda}\}_{\lambda \in \Lambda}$ of $X$,
there exists a partition of unity $\{ \psi_{\lambda} \colon X \to \mathbf{R} \mid \lambda \in \Lambda \}$ 
subordinate to $\mathcal U$.
\end{thm}
\begin{proof}
Since $TX$ is a paracompactum by Proposition \ref{prp:paracompactum},
there exists a locally finite $D$-open cover $\{U^{\prime}_{\lambda}\}$ such that 
$U^{\prime}_{\lambda} \subset \overline{U^{\prime}}_{\lambda} \subset U_{\lambda}$ holds 
for each $\lambda \in \Lambda$.
Then we get a smooth map $f_{\lambda} \colon X \to I$ such that supp$f_{\lambda} \subset U_{\lambda}$ 
and $f_{\lambda} \equiv 1$ on $\overline{U^{\prime}}_{\lambda}$.
Define a smooth map $\psi_{\lambda} \colon X \to \mathbf{R}$ by 
\[
\psi_{\lambda}(x)=
\frac{f_{\lambda}(x)}
{
\sum_{\lambda \in \Lambda} f_{\lambda}(x)
}.
\]
Then 
$
\{\psi_{\lambda} \colon X \to \mathbf{R} \mid \lambda \in \Lambda \}
$
is a partition of unity subordinate to $\mathcal{U}$.
\end{proof}
%%%%%%%%%%%%%%%%%%%%%%%%%%%%%%%%%%%%%%%%%%%%%%%%%%%%%%%%%%%%%%%%%%%%%%%%%%%%%%%%%%%%%%%%%%%%%%%%%%%%%%%%%%%%%%%%%%
%%%%%%%%%%%%%%%%%%%%%%%%%%%%%%%%%%%%%%%%%%%%%%%%%%%%%%%%%%%%%%%%%%%%%%%%%%%%%%%%%%%%%%%%%%%%%%%%%%%%%%%%%%%%%%%%%%
%%%%%%%%%%%%%%%%%%%%%%%%%%%%%%%%%%%%%%%%%%%%%%%%%%%%%%%%%%%%%%%%%%%%%%%%%%%%%%%%%%%%%%%%%%%%%%%%%%%%%%%%%%%%%%%%%%

%%%%%%%%%%%%%%%%%%%%%%%%%%%%%%%%%%%%%%%%%%%%%%%%%%%%%%%%%%%%%%%%%%%%%%

\end{document}